\documentclass[11pt,reqno]{amsart}

\usepackage{amsmath}
\usepackage{amsfonts}
\usepackage{amssymb}
\usepackage{amsthm}
\usepackage{graphicx}
\usepackage{epstopdf}
\usepackage{enumitem}
\usepackage{geometry}
\usepackage{hyperref}
\usepackage{tikz}
\usepackage[all]{xy}

\newtheorem{theorem}{Theorem}
\newtheorem{definition}[theorem]{Definition}

\newtheorem{proposition}[theorem]{Proposition}
\newtheorem{remark}[theorem]{Remark}
\newtheorem{corollary}[theorem]{Corollary}
\newtheorem{lemma}[theorem]{Lemma}

\newcommand{\Hom}{\operatorname{\mathbb{H}om}}

\newcommand{\N}{\mathbb{N}}

\newcommand{\R}{R^\omega}

\begin{document}

\title{Groups associated to II$_1$-factors}

\author{Nathanial P. Brown}
\address{Penn State University, USA}
\email{nbrown@math.psu.edu}

\author{Valerio Capraro}
\address{University of Neuchatel, Switzerland}
\thanks{N.B. was supported by NSF-0856197, V.C. by Swiss SNF Sinergia project CRSI22-130435.}
\email{valerio.capraro@unine.ch}

\date{}

\maketitle

\begin{abstract} We extend recent work of the first named author, constructing a natural $\Hom$ semigroup associated to any pair of II$_1$-factors.  This semigroup always satisfies cancelation, hence embeds into its Grothendieck group. When the target is an ultraproduct of a McDuff factor (e.g., $R^\omega$), this Grothendieck group turns out to carry a natural vector space structure; in fact, it is a Banach space with natural actions of outer automorphism groups. 
\end{abstract}

\tableofcontents

\section{Introduction and main results}

Let $\omega\in\beta(\mathbb N)\setminus\mathbb N$ be a free
ultrafilter on the natural numbers and $R^\omega$ be the
corresponding ultrapower of the hyperfinite $II_1$-factor $R$.  For a separable factor $N$ the space of unital embeddings into $R^\omega$ modulo inner automorphisms, denoted $\Hom(N, \R)$, has a surprisingly rich structure. (When it is nonempty, as Connes' famous embedding problem asks \cite{Co}.) For example, in \cite{Br} it was shown to be a complete metric space with ``convex-like" structure, meaning that one could define convex combinations even though $\Hom(N, \R)$ isn't defined as a subset of a vector space.\footnote{For the original axioms of a convex-like structure we refer the reader to \cite[Definition 2.1]{Br}. These axioms have been simplified in Corollary 12 in \cite{Ca-Fr}.}  During a lecture in Nottingham the first author posed the problem of constructing a vector space embedding and two suggestions were made.  Aaron Tikuisis proposed a universal vector space construction that could be used on any abstract convex-like space.  The second author and Tobias Fritz independently had a similar idea,  showing in \cite{Ca-Fr} that everything works and, even better, one can realize any convex-like space as a closed convex set in a Banach space.  

The second suggestion in Nottingham was made by Ilijas Farah who proposed using the fundamental group of $\R$ and a Grothendieck construction to produce a vector-space embedding.  This is the path we follow here.  It is quite instructive to reduce this idea to its essence    and start in full generality.  Adding structure to the algebras leads to additional structure on the $\Hom$ spaces and only in the case that the target is an ultraproduct of a McDuff factor can we prove that one gets a vector space (even a Banach space).  Indeed, it turns out that Farah's very natural and beautiful idea is surprisingly subtle to prove, depends (as far as we can tell) in a crucial way on the special structure of ultraproducts of McDuff factors and ought not be expected to hold in the absence of similar structures.  

In more detail, let $N$ and $M$ be II$_1$-factors, $H$ be a separable, infinite-dimensional Hilbert space and $B(H)$ denote the bounded linear operators on $H$.

\begin{definition} We let $M^{\infty} \subset B(H)\bar{\otimes} M$ be the compact ideal (i.e., the algebraic ideal generated by projections of finite trace) and $\Hom(N, M^{\infty})$ be the collection of $*$-homomorphisms $\pi \colon N \to M^{\infty}$ modulo inner automorphisms of $B(H)\bar{\otimes} M$, i.e., $[\pi_1] = [\pi_2] \Longleftrightarrow \exists$ unitary $u \in B(H)\bar{\otimes} M$ such that $\pi_1 = \mathrm{Ad}u \circ \pi_2$. 
\end{definition} 
 
$\Hom(N, M^{\infty})$ carries a natural ``topology of point-wise convergence" where $[\pi_n] \to [\pi]$ means there exist representatives $\tilde{\pi}_n \sim \pi_n$ such that $\tilde{\pi}_n (x) \to \pi(x)$ in the $\sigma$-weak topology, for all $x \in N$.  Just as with K-theory or (using the Busby picture of) $\mathrm{Ext}$-theory for C$^*$-algebras, one defines a natural addition on $\Hom(N, M^{\infty})$ and we thus get a topological semigroup, where the zero homomorphism plays the role of the neutral element. Predictably, the outer automorphism groups of $N$ and $B(H) \bar{\otimes} M$ act continuously by pre- and post-composition, respectively, yielding  topological dynamical systems. Less obvious is the fact that $\Hom(N, M^{\infty})$ always satisfies cancellation, hence embeds into its Grothendieck group.  

\begin{definition}  Let $\mathcal{G} (N,M)$ denote the Grothendieck group of $\Hom(N, M^{\infty})$, equipped with the canonical actions of $\mathrm{Out}(N)$ and $\mathrm{Out}(M \bar{\otimes} B(H))$. 
\end{definition} 

Section 2 is devoted to proving the assertions above.  In section 3 we turn to fundamental groups.  That is, since elements of the fundamental group $\mathcal{F}(M)$ correspond to trace-scaling automorphisms of $B(H) \bar{\otimes} M$, one can ask whether $\Hom(N, M^{\infty})$ carries an action of this important invariant.  Examples of Popa and Vaes show it doesn't (at least canonically) in general, since there need not be a group homomorphism $\mathcal{F}(M) \hookrightarrow \mathrm{Out}(B(H) \bar{\otimes} M)$ (cf.\ \cite{PV}).  However, if $N$ is separable and $M$ is the ultraproduct of a McDuff factor, we will construct a particularly nice action of $\mathbb{R}_+$ on  $\Hom(N, M^{\infty})$.    

When $\mathcal{F}(M) = \mathbb{R}_+$ and there is a group homomorphism $\delta\colon \mathbb{R}_+ \to \mathrm{Out}(B(H) \bar{\otimes} M)$, one is tempted to extend it to an action of $\mathbb{R}$ on $\mathcal{G} (N,M)$ that produces a vector space structure.  Unfortunately, there is no reason to expect that for $s,t \in \mathbb{R}_+$ and $[\pi] \in \Hom(N, M^{\infty})$ we should have $$(s + t) [\pi] = s[\pi] + t[\pi].$$  Indeed, we rather doubt such distributivity  holds in general.  However, we observe that in the case $N$ is separable and $M$ is an ultraproduct of a McDuff factor, we do have $(s + t) [\pi] = s[\pi] + t[\pi]$ and this turns $\mathcal{G} (N,M)$ into a vector space. (One part of the proof, surely known to algebraists but included for the reader's convenience, is relegated to an appendix.) 

The main results of this paper are summarized as follows.  

\begin{theorem} For arbitrary II$_1$-factors $N$ and $M$, $\mathcal{G} (N,M)$ is a topological group with canonical actions of $\mathrm{Out}(N)$ and $\mathrm{Out}(B(H) \bar{\otimes} M)$.  

If $N$ is separable and $M = X^\omega$ for some McDuff factor $X$, then $\mathcal{F}(M) = \mathbb{R}_+$ acts on $\mathcal{G} (N,M)$ (via a homomorphism $\delta \colon \mathcal{F}(M) \to \mathrm{Out}(B(H) \bar{\otimes} M)$)   and extends to all of $\mathbb{R}$ yielding a vector space structure.  In fact, following \cite{Ca-Fr}, the topology on $\Hom(N, M)$ can be realized by a norm on $\mathcal{G} (N,M)$ yielding a Banach space. 
\end{theorem} 

{\bf Acknowledgement:} The authors would like to thank the referee for their fastidious proofreading  and several helpful suggestions that improved the exposition of this paper.

\section{Constructing the group $\mathcal{G} (N,M)$} 

With notation as in the introduction, our first task is to describe the semigroup structure on $\Hom(N, M^{\infty})$.

\begin{definition}\label{def:sum}
If $[\phi],[\psi]\in \Hom(N, M^{\infty})$, we define $$[\phi] + [\psi] := [\tilde{\phi} + \psi],$$ where $\tilde{\phi}$ is a representative of $[\phi]$ with the property that $\tilde{\phi} (1) \perp \psi(1)$.  
\end{definition}

Since $\phi(1)$ and $\psi(1)$ have finite trace, we can always find $\tilde{\phi}$ by simply choosing a unitary $u \in M \bar{\otimes} B(H)$ such that $u\phi(1) u^* \perp \psi(1)$ and declaring $\tilde{\phi} = \mathrm{Ad}u \circ \phi$.

\begin{lemma}\label{semigroup}
The operation $+$ is well-defined and makes $\Hom(N, M^{\infty})$ an abelian semigroup.
\end{lemma}

\begin{proof}  To see that $+$ is well defined, first suppose we have two representatives $\phi_1$ and $\phi_2$ of $[\phi]$, each with the property that $\phi_i(1) \perp \psi(1)$ ($i = 1,2$).  In this case, there is a unitary $u$ such that $\phi_2 = \mathrm{Ad}u \circ \phi_1$.  Choose a partial isometry $w$ such that $w^* w = 1 - \phi_1(1), ww^* = 1 - \phi_2(1)$ and $w \psi(1) = \psi(1) w = \psi(1)$. (This is possible because $1 - \phi_1(1)$ and $1 - \phi_2(1)$ are infinite projections dominating the finite projection $\psi(1)$.) Define a new unitary $$v := u \phi_1(1) + w$$ and a routine calculation shows $\phi_2 + \psi = \mathrm{Ad}v \circ (\phi_1 + \psi)$. 

Showing $+$ is independent of the representative of $[\psi]$ is similar, thus $+$ is well defined.  Checking commutativity and associativity are now routine exercises, so we leave the details to the reader. 
\end{proof}

\begin{remark}\label{rem:topologytometric}
{\rm  The ``point-wise convergence" topology on $\Hom(N, M^{\infty})$ can be viewed via pseudometrics, in the case $N$ is countably generated by contractions $\{ a_i \}$.  For example, an $\ell^2$ pseudometric such as $$
d([\phi],[\psi])=\inf_{u\in
U(M)}\left(\sum_{n=1}^\infty||\phi(\frac{1}{2^{n}}a_n)-u\psi(\frac{1}{2^{n}}a_n)u^*||_2^2\right)^{\frac{1}{2}},
$$ is easily seen to generate this topology, as would similar $\ell^p$ versions.  In some cases, like when $M$ is an ultraproduct, $d( \cdot, \cdot)$ becomes an honest metric (cf.\ \cite[Theorem 3.1]{Sh} and \cite[Proposition 3.1]{Br2}).}
\end{remark}

\begin{lemma}\label{topological}
$(\Hom(N, M^{\infty}),+)$ is a topological monoid with \begin{color}{red}
actions of $Out(N)$ and $Out(M)$ via continuous homeomorphisms.\end{color}
\end{lemma}

\begin{proof} The zero homomorphism $N \to M^{\infty}$ evidently gives rise to an identity element in $\Hom(N, M^{\infty})$, hence we have a monoid. 

To see that $+$ is continuous, suppose $[\phi_n]\rightarrow[\phi]$ and
$[\psi_n]\rightarrow[\psi]$. Changing representatives if necessary, we may assume $\phi(1)\perp\psi(1)$, $\phi_n \to \phi$ and $\psi_n \to \psi$ (point-$\sigma$-weakly). Let $u_n$ be a sequence of
unitaries  such that
$u_n^*\phi_n(1)u_n\perp\psi_n(1)$.  Since $\phi_n(1)$ and $\psi_n(1)$ are asymptotically orthogonal already, we may further assume that $u_n p u_n^* \to p$ $\sigma$-weakly, for every finite projection $p \in M \bar{\otimes} B(H)$. It follows that
$[\phi_n]+[\psi_n]=[u_n^*\phi_nu_n+\psi_n]$ and  $(u_n^*\phi_nu_n+\psi_n) \to (\phi + \psi)$ point-$\sigma$-weakly,  so our monoid is topological. 

Actions of the outer automorphism groups $Out(N)$ and $Out(B(H) \bar{\otimes} M)$ are given by pre- and post-composition, respectively: $\alpha.[\phi] = [\phi \circ \alpha^{-1}]$ for all $\alpha \in Out(N)$ and $\beta.[\phi] = [\beta \circ \phi]$ for all $\beta \in Out(B(H) \bar{\otimes} M)$.  Proving these two actions are monoidal homeomorphisms are very similar, so we only do it for $Out(N)$.  

It is routine to check that $\alpha.[\phi] = [\phi \circ \alpha^{-1}]$ is well defined, since different representatives of $\alpha \in Out(N)$ differ by inner automorphisms.  As for continuity, \begin{color}{red}choosing the right representatives for the classes $[\phi_n]$ and $[\phi]$, one has\end{color}
\begin{align*}
[\phi_n]\rightarrow[\phi]\Leftrightarrow\\
&\Leftrightarrow\phi_n(x)\rightarrow\phi(x),
\forall x\in
N\\
&\Leftrightarrow\phi_n(\alpha^{-1}(x))\rightarrow\phi(\alpha^{-1}(x)),\forall x\\
&\Leftrightarrow(\phi_n\circ\alpha^{-1})(x)\rightarrow(\phi\circ\alpha^{-1})(x),\forall
x\in N\\
&\Leftrightarrow[\phi_n\circ\alpha^{-1}]\rightarrow[\phi\circ\alpha^{-1}]\\
&\Leftrightarrow\alpha.[\phi_n]\rightarrow\alpha.[\phi].
\end{align*} 
Similarly, a calculation shows $\alpha.(\cdot)$ preserves $+$: 
\begin{align*}
\alpha.([\phi]+[\psi])=\\
&=\alpha.[u\phi u^*+\psi]\\
&=[(u\phi
u^*+\psi)\circ\alpha^{-1}]\\
&=[u(\phi\circ\alpha^{-1})u^*+(\psi\circ\alpha^{-1})]\\
&=\alpha.[\phi]+\alpha.[\psi]
\end{align*}
Finally, it is clear that $\alpha.(\cdot)$ is a bijection with (continuous) inverse $\alpha^{-1}.(\cdot)$, so the proof is complete. 
\end{proof} 

Now we move towards the cancelation property. We need the following

\begin{lemma}\label{lem:mvnequivalence}
Given a morphism $\phi:N\rightarrow M^{\infty}$ and
projections $p,q\in\phi(N)'\cap M^{\infty}$, with $p,q\leq\phi(1)$. The
following are equivalent:
\begin{enumerate}
\item There exists a partial isometry $v\in \phi(1) M^{\infty} \phi(1)$ such that $vv^*=q$,
$v^*v=p$ and $v\phi(x)v^*=q\phi(x)$, for all $x\in N$.
\item $p\sim q$ in $\phi(N)'\cap \phi(1) M^{\infty} \phi(1)$.
\item $[p\phi]=[q\phi]$, where $p\phi:N\rightarrow M$ is defined by
$x\rightarrow p\phi(x)$.
\end{enumerate}
\end{lemma}

\begin{proof}
$1)\Rightarrow2)$. It suffices to show that $v$ commutes with
$\phi(x)$, for all $x\in N$. Indeed
\begin{color}{red}
\begin{align*}
v^*\phi(x)\\
&=v^*q\phi(x)\\
&=v^*v\phi(x)v^*\\
&=p\phi(x)v^*\\
&=\phi(x)v^*
\end{align*}
\end{color}

$2)\Rightarrow3)$. Choose partial isometries $v \in\phi(N)'\cap \phi(1) M^{\infty} \phi(1)$ and $w \in \phi(N)'\cap \phi(1) M^{\infty} \phi(1)$
such that $v^*v=p, vv^*=q$, $w^*w=p^\perp$ and $ww^*=q^\perp$. (It is possible to find $w$ since $\phi(N)'\cap \phi(1) M^{\infty} \phi(1)$ is a finite von Neumann algebra.) Hence
$u=v+w\in\phi(N)'\cap \phi(1) M^{\infty} \phi(1)$ is a unitary and
$$
up\phi(x)u^*=upu^*\phi(x)=q\phi(x).
$$
Extending $u$ to a unitary in $B(H) \bar{\otimes} M$ we see  $[p\phi]=[q\phi]$.

$3)\Rightarrow1)$. Choose a unitary $u\in B(H) \bar{\otimes} M$ such that
$up\phi(x)u^*=q\phi(x)$, for all $x\in N$. Define $v=up$ and, using the assumption that $p,q \leq \phi(1)$, one can check this does the trick. 
\end{proof}

\begin{proposition}\label{cancellation}
$\Hom(N, M^{\infty})$ has cancellation, i.e., 
$$
[\rho]+[\phi]=[\rho]+[\psi]\Rightarrow[\phi]=[\psi].
$$
\end{proposition}

\begin{proof}
We may assume that $\phi(1) = \psi(1)$ (since they have the same trace) and $\phi(1) \perp \rho(1)$. Let
$u\in M \bar{\otimes} B(H)$ be a unitary such that
$\rho+\phi=u(\rho+\psi)u^*$ and set $p=\rho(1)$ and
$q=u\rho(1)u^*$. Then $p(\rho+\phi)=\rho$ and
$q(\rho+\phi)= q(u(\rho+\psi)u^*) = u\rho u^*$. It follows that
$[p(\rho+\phi)]=[q(\rho+\phi)]$ and so, by Lemma \ref{lem:mvnequivalence}, $p$
and $q$ are Murray-von Neumann equivalent inside $((\rho+\phi)(N))'\cap (\rho+\phi)(1)M(\rho+\phi)(1)$; hence, so are $(\rho+\phi)(1) - p = \phi(1)$ and $(\rho+\phi)(1) - q = u\psi(1)u^*$. Therefore, using
once again Lemma \ref{lem:mvnequivalence}, we get
$$
[\phi]=[\phi(1)(\rho+\phi)]=[u\psi(1)u^*(u(\rho+\psi)u^*)]=[u\psi u^*]=[\psi].
$$
\end{proof}

Thanks to cancelation, $\Hom(N, M^{\infty})$ embeds into its
Grothendiek group $\mathcal G(N,M)$. Note that
$\mathcal G(N,M)$ carries a canonical topology, given by the
quotient of the product topology. As one would hope, the
main properties of $\Hom(N, M^{\infty})$ are inherited by $\mathcal
G(N,M)$.

\begin{proposition}\label{quotient}
The group $\mathcal G(N,M)$ is a topological abelian group. Moreover
$Out(N)$ and $Out(M)$ act on $\mathcal G(N,M)$ via \begin{color}{red}
continuous group homeomorphisms.\end{color}
\end{proposition}

\begin{proof}
$\mathcal G(N,M)$ is an abelian group. In order to prove that the
sum is continuous, let us fix a piece of notation:
$[([\phi],[\psi])]_{\mathcal G}$ denotes the class of
$([\phi],[\psi])\in\Hom(N,M^\infty )\times \Hom(N,M^\infty)$ with respect to
the Grothendieck equivalence relation, which will be denoted by $\sim_{\mathcal G}$. Now suppose that $
[([\phi_n],[\psi_n])]_{\mathcal
G}\rightarrow[([\phi],[\psi])]_{\mathcal G}$ and
$[([\beta_n],[\gamma_n])]_{\mathcal
G}\rightarrow[([\beta],[\gamma])]_{\mathcal G}$.
\begin{color}{red}
This means that there are representatives $([\phi_n]',[\psi_n]')\sim_{\mathcal G}([\phi_n],[\psi_n])$, $([\phi]',[\psi]')\sim_{\mathcal G}([\phi],[\psi])$, $([\beta_n]',[\gamma_n]')\sim_{\mathcal G}([\beta_n],[\gamma_n])$, and $([\beta]',[\gamma]')\sim_{\mathcal G}([\beta],[\gamma])$  such that
$$
([\phi_n]',[\psi_n]')\to([\phi]',[\psi]')\qquad\text{and}\qquad([\beta_n]',[\gamma_n]')\to([\beta]',[\gamma]')
$$
in the product topology of $\mathbb Hom(N,M^\infty)\times\mathbb Hom(N,M^\infty)$. Thus, there are representatives $[\tilde{\cdot}]'$ of $[\cdot]'$ such that
$$
[\tilde\phi_n]'\to[\tilde\phi]'\qquad[\tilde\psi_n]'\to[\tilde\psi]'\qquad[\tilde\beta_n]'\to[\tilde\beta]'\qquad[\tilde\gamma_n']\to[\tilde\gamma]'
$$
in $\mathbb Hom(N,M^\infty)$. By Lemma \ref{topological}, it follows that
$$
([\tilde\phi_n]'+[\tilde\beta_n]',[\tilde\psi_n]'+[\tilde\gamma_n]')\to([\tilde\phi]'+[\tilde\beta]',[\tilde\psi]'+[\tilde\gamma]')
$$
Therefore, it suffices to show that
\begin{align}\label{eq:first}
([\tilde\phi_n]'+[\tilde\beta_n]',[\tilde\psi_n]'+[\tilde\gamma_n]')\sim_{\mathcal G}([\phi_n]+[\beta_n],[\psi_n]+[\gamma_n])
\end{align}
and 
\begin{align}\label{second}
([\tilde\phi]'+[\tilde\beta]',[\tilde\psi]'+[\tilde\gamma]')\sim_{\mathcal G}([\phi]+[\beta],[\psi]+[\gamma])
\end{align}
Since the proofs are very similar, we show only (\ref{eq:first}). First observe that we can take out all the tilda's without modifying the equivalence classes, then, by the very definition of the Grothendieck construction, let $[\rho]$ and $[\sigma]$ be such that
$$
[\phi_n]'+[\psi_n]+[\rho]=[\phi_n]+[\psi_n]'+[\rho] \qquad\text{and}\qquad[\beta_n]'+[\gamma_n]+[\sigma]=[\beta_n]+[\gamma_n]'+[\sigma]. 
$$
One can now obtain (\ref{eq:first}) just summing these two equalities.
\end{color}

The actions of $Out(N)$ and $Out(M\bar{\otimes}B(H))$ are defined in the obvious way and checking they are well defined and yield continuous group actions is a routine exercise left to the reader. 
\end{proof}

The group $\mathcal G(N,M)$ may be trivial, for instance if $N$ has property (T) and $M$ has the Haagerup property (cf. \cite{CJ}). At the other extreme, if $M=\mathbb R^\omega$ and $N \subset M$ is any non-hyperfinite subfactor, then $\mathcal G(N,M)$ is nonseparable; and  $\mathcal G(N,M)$ is a point if $N$ is hyperfinite (see \cite{Br}).  It would be nice to find examples that lie between these extremes.

\section{Fundamental groups} 

Recall that the fundamental group $\mathcal{F}(M)$ is the set of $t > 0$ such that $M \cong M_t$, where $M_t = p_t M^{\infty} p_t$ for some projection $p_t$ of trace $t$.  Elements of $\mathcal{F}(M)$ give rise to trace-scaling automorphisms of $B(H) \bar{\otimes} M$, but there need not be a group homomorphism $\delta \colon \mathcal{F}(M) \to Out(B(H) \bar{\otimes} M)$ (cf.\ \cite{PV}).  Of course, when such a homomorphism exists we get actions of $\mathcal{F}(M)$ on $\Hom(N, M^{\infty})$ and $\mathcal G(N,M)$.  In this section we specialize to the case $N$ is separable and $M = X^\omega$ for some McDuff factor $X$, then construct a particularly nice action of $\mathcal{F}(X^\omega) = \mathbb{R}_+$ on $\Hom(N, M^{\infty})$. 

Let $X$ be a McDuff $II_1$-factor and fix a *isomorphism $\Phi:R\bar\otimes X\to X$.
Denote by $\Phi_\omega:(R\bar\otimes X)^\omega\to X^\omega$ the component-wise *isomorphism induced by $\Phi$. Since II$_1$-factors always have a unique trace, we use $\tau$ to denote them all. 

\begin{definition}
Let $p\in X^\omega$ be a projection such that $\Phi_\omega^{-1}(p)$ has the form $\tilde p\otimes 1=(\tilde p_n\otimes 1)_n\in(R\otimes X)^\omega$, with $\tau(\tilde p_n)=\tau(\tilde p)=\tau(p)$. A standard isomorphism $\theta:X^\omega\to pX^\omega p$ is any isomorphism gotten in the following way. Fix isomorphisms $\alpha_n:R\to\tilde p_nR\tilde p_n$ and let $\theta_n:=\alpha_n\otimes Id:R\bar\otimes X\to \tilde p_nR\tilde p_n\bar\otimes X$. Define $\theta$ to be the isomorphism on the right hand side of the following diagram
$$
\xymatrix{\ell^\infty(R\bar\otimes X)\ar[r]\ar[d]^{\oplus_{\mathbb N}\theta_n} & (R\bar\otimes X)^\omega\ar[r]\ar[d]^{_\omega\theta} & X^\omega\ar[d]^{\theta}\\
\ell^\infty((\tilde p_n\otimes1)(R\bar\otimes X)(\tilde p_n\otimes1))\ar[r] & (\tilde p\otimes1)(R\bar\otimes X)^\omega(\tilde p\otimes1)\ar[r] & pX^\omega p}
$$
where the horizontal left-hand side arrows are the projections onto the quotient, the horizontal right-hand side arrows are the ultrapower isomorphisms $\Phi_\omega$, and the isomorphism $_\omega\theta$ is the one obtained by imposing commutativity on the left-half of the diagram.
\end{definition}

Since a McDuff $II_1$-factor has full fundamental group, for all $t\in(0,1),$ there is a standard isomorphism $\theta_t:X^\omega\to p_tX^\omega p_t$, where $p_t\in X^\omega$ is a projection of trace $t$ such that $\Phi_\omega^{-1}(p_t)$ has the form $\tilde p_t\otimes1\in (R\bar\otimes X)^\omega$.\\

The following Lemma is very similar to Proposition 3.1.2 in \cite{Br} and it is one of the main technical tools that we need.

\begin{lemma}\label{lem:equivalence}
Let $p,q\in X^\omega$ be projections of the same trace as needed to define standard isomorphisms $\theta_p,\theta_q$. For all separable von Neumann subalgebras $M_1\subseteq X^\omega$, there is a partial isometry $v_1\in X^\omega$ such that $v_1^*v_1=p$, $v_1v_1^*=q$ and
$$
v_1\theta_p(x)v_1^*=\theta_q(x)\qquad\qquad\text{ for all } x\in M_1
$$
\end{lemma}
\begin{proof}
With the obvious notation, consider the following commutative diagram

$$
\xymatrix{\ell^\infty((\tilde p_n\otimes1)(R\bar\otimes X)(\tilde p_n\otimes1)))\ar[r]\ar[d]^{\oplus\theta_{p_n}^{-1}} & (\tilde p\otimes1)(R\bar\otimes X)^{\omega}(\tilde p\otimes1)\ar[r]\ar[d]^{(_\omega\theta_p)^{-1}} & pX^\omega p\ar[d]^{\theta_p^{-1}}\\
\ell^\infty(R\bar\otimes X)\ar[r]\ar[d]^{\oplus_{\mathbb N}\theta_{q_n}} & (R\bar\otimes X)^\omega\ar[r]\ar[d]^{_\omega\theta_q} & X^\omega\ar[d]^{\theta_q}\\
\ell^\infty((\tilde q_n\otimes1)(R\bar\otimes X)(\tilde q_n\otimes1))\ar[r] & (\tilde q\otimes1)(R\bar\otimes X)^\omega(\tilde q\otimes1)\ar[r] & qX^\omega q}
$$
Consider $\Phi_\omega^{-1}(M_1)\subseteq(R\bar\otimes X)^\omega$. In the left-half of the previous diagram, we may apply Proposition 3.1.2 in \cite{Br} to $\Theta= _\omega\theta_q\circ(_\omega\theta_p)^{-1}$ and $M=\Phi_\omega^{-1}(M_1)$, since all isomorphisms act only on the hyperfinite $II_1$-factor $R$. Thus, there is a partial isometry $v\in(R\bar\otimes X)^\omega$ such that $v^*v=\tilde p\otimes1$, $vv^*=\tilde q\otimes1$ and
\begin{align}\label{eq:equivalence}
v(_\omega\theta_p(x))v^*=_\omega\theta_q(x)\qquad\qquad\text{for all } x\in\Phi_\omega^{-1}(M_1)
\end{align}
Define $v_1=\Phi_\omega(v)$ and one can verify that it works.
\end{proof}

Let $t\in(0,1)$ and let $p_t\in X^\omega$ be a projection of trace $t$ as needed to define a standard isomorphism $\theta_t:X^\omega\to p_tX^\omega p_t$. Let us recall the construction of a trace-scaling automorphism $\Theta_t$ of $B(H)\bar\otimes X^\omega$, since it will be helpful in the proof of Proposition \ref{prop:independent}. More details can be found in \cite{Ka-Ri2}, Proposition 13.1.10.

Let $\{e_{jj}\}\subseteq B(H)$ be a countable family of orthogonal one-dimensional projections such that $\sum e_{jj}=1$ and let $e_{jk}$ be partial isometries mapping $e_{jj}$ to $e_{kk}$. Define $f_{jk}=e_{jk}\otimes1\in B(H)\bar\otimes X^\omega$.  We know that $f_{11}(B(H)\bar\otimes X^\omega)f_{11}$ is *isomorphic to $X^\omega$ and that $\tau_{\infty}$ is normalized in such a way that $\tau_{\infty}(f_{11})=1$. Thus we can look at $p_t$ as a projection in $f_{11}(B(H)\bar\otimes X^\omega)f_{11}$ with trace $t$ and, for simplicity, let us denote it by $g_{11}$. Let $g_{jj}$ be a countable family of orthogonal projections, each of which is equivalent to $g_{11}$, such that $\sum g_{jj}=1\in B(H)\bar\otimes X^\omega$ and extend the family $\{g_{jj}\}$ to a system of matrix units $\{g_{jk}\}$ of $B(H)\bar\otimes X^\omega$ adding appropriate partial isometries. Now, for any algebra $A \subset B(K)$, denote by $\aleph_0\otimes A$ the algebra of countably infinite matrices with entries in $A$ that define bounded operators on $\oplus_\N K \cong H \otimes K$.  The isomorphism $\theta_t:X^\omega\to p_tX^\omega p_t$ can be seen as an isomorphism $\theta_t:f_{11}(B(H)\bar\otimes X^\omega)f_{11}\to p_{t}(B(H)\bar\otimes X^\omega)p_t$ and then it gives rise to an isomorphism
$$
\aleph_0\otimes\theta_t:\aleph_0\otimes (f_{11}(B(H)\bar\otimes X^\omega)f_{11})\to\aleph_0\otimes(p_{t}(B(H)\bar\otimes X^\omega)p_t)
$$
Now, let $G$ be the matrix in $\aleph_0\otimes (f_{11}(B(H)\bar\otimes X^\omega)f_{11})$ having the unit in the position $(1,1)$ and zeros elsewhere. Then $(\aleph_0\otimes\theta_t)(G)$ is the matrix in $\aleph_0\otimes (p_t(B(H)\bar\otimes X^\omega)p_t)$ having the unit in the position $(1,1)$ and zeros elsewhere. Now, take isomorphisms
$$
\phi_1:B(H)\bar\otimes X^\omega\to\aleph_0\otimes (f_{11}(B(H)\bar\otimes X^\omega)f_{11})\qquad\phi_2:B(H)\bar\otimes X^\omega\to\aleph_0\otimes (p_t(B(H)\bar\otimes X^\omega)p_t)
$$
such that $\phi_1(f_{11})=G$ and $\phi_2(g_{11})=(\aleph\otimes\theta_t)(G)$. Define
$$
\Theta_t=\phi_2^{-1}\circ(\aleph_0\otimes\theta_t)\circ\phi_1
$$
Is is easily checked that $\tau_\infty(\Theta_t(x))=t\tau_\infty(x)$, for all $x$.

\begin{remark}\label{rem:matrixrepr}
{\rm For the sequel, it is important to stress the fact that $\Theta_t$ is nothing but the isomorphism obtained by writing $B(H)\bar\otimes X^\omega$ as an algebra of countably infinite matrices and letting $\theta_t$ act on each component. Therefore, if we want to prove that two isomorphisms $\Theta_t^{(1)}$ and $\Theta_t^{(2)}$ constructed in such a fashion are unitarily equivalent, it suffices to find unitaries mapping $\theta_t^{(1)}$ to $\theta_t^{(2)}$ and the matrix units used in the first representation of $B(H)\bar\otimes X^\omega$ as a matrix algebra to the matrix units used in the second representation.}
\end{remark}

\begin{definition}\label{def:scalarmultiplication}
\begin{color}{red}Let $t\in(0,1]$ and $[\phi]\in\mathbb Hom(N,(X^\omega)^\infty)$. \end{color}We define
$$
t[\phi]=[\Theta_t\circ\phi]
$$
\end{definition}

Remark \ref{rem:matrixrepr} is important because now we need to prove that the definition of $t[\phi]$ depends only on $t$ and $[\phi]$ and is independent of $\Theta_t$.

\begin{proposition}\label{prop:independent}
Let $t\in(0,1]$, $p_t^{(i)}\in X^\omega$, $i=1,2$, be two projections
of trace $t$ and $\theta_t^{(i)}:X^\omega\rightarrow
p_t^{(i)}X^\omega p_t^{(i)}$ be two standard isomorphisms. Then $\Theta_t^{(1)}\circ\phi$ is unitarily
equivalent to $\Theta_t^{(2)}\circ\phi$.
\end{proposition}

\begin{proof}
Let us start with an observation. The image $\phi(N)$ a priori belongs to $B(H)\bar\otimes X^\omega$, but since $\tau_\infty(\phi(1))<\infty$, we can twist it by a unitary and suppose that $\phi(N)\subseteq M_n(\mathbb C)\otimes X^\omega$, for some $n>\tau_\infty(\phi(1))$. Now, for all $j=1,\ldots,n$, let
\begin{color}{red}
$$
M_j=(e_{jj}\otimes1)\phi(N)(e_{jj}\otimes1)\subseteq(e_{jj}\otimes1)(B(H)\bar\otimes X^\omega)(e_{jj}\otimes1)\cong X^\omega
$$
\end{color}
Since $p_t^{(1)}$ is equivalent to $p_t^{(2)}$ and $(p_t^{(1)})^\perp$ is equivalent to $(p_t^{(2)})^\perp$, in Lemma \ref{lem:equivalence} we may find a unitary $u_i\in X^\omega$ such that
$$
(e_{jj}\otimes u_j)((e_{jj}\otimes\theta_t^{(1)})(x))(e_{jj}\otimes u_j)=(e_{jj}\otimes\theta_t^{(2)})(x)\qquad\text{ for all } x\in M_j
$$
where $e_{jj}\otimes\theta_t^{(1)}$ stands for the endomorphism obtained letting $\theta_t^{(1)}$ act only on $f_{jj}(B(H)\bar\otimes X^\omega)f_{jj}$.
Since the partial isometries $e_{jj}\otimes u_j$ act on orthogonal subspaces, we may extend them \emph{all together} to a unitary $u\in B(H)\bar\otimes X^\omega$ such that
\begin{color}{red}
$$
u((e_{jj}\otimes\theta_t^{(1)})(x))u^*=(e_{jj}\otimes\theta_t^{(2)})(x)\qquad\text{ for all } j=1,\ldots,n \text{ and for all } x\in M_j
$$
\end{color}
Set $e_n=\sum_{j=1}^ne_{jj}$. We have
$$
u((e_n\otimes\theta_t^{(1)})(x))u^*=(e_n\otimes\theta_t^{(2)})(x)\qquad\text{for all } x\in(e_n\otimes1)\phi(N)(e_n\otimes1)=\phi(N)
$$
Now observe that the matrix units $\left\{f_{jk}^{(1)}\right\}$ and $\left\{f_{jk}^{(2)}\right\}$ used to construct $\Theta_t^{(1)}$ and $\Theta_t^{(2)}$ are unitarily equivalent, since the projections on the diagonal have the same trace. Therefore, also the matrix units $\left\{uf_{jk}^{(1)}u^*\right\}$ and $\left\{f_{jk}^{(2)}\right\}$ are unitarily equivalent. Let $w\in B(H)\bar\otimes X^\omega$ be a unitary such that
$$
w(uf_{jk}^{(1)}u^*)w^*=f_{jk}^{(2)}\qquad\qquad\text{ for all } j,k\in\mathbb N
$$
The unitary $w$ then twists the matrix units $uf_{jk}^{(1)}u^*$ into the matrix units $f_{jk}^{(2)}$ and it twists $u((e_n\otimes\theta_t^{(1)})(x))u^*$ to $(e_n\otimes\theta_t^{(2)})(x)$, for all $x\in\phi(N)$. Therefore, by Remark \ref{rem:matrixrepr},
$$
wu\Theta_t^{(1)}(x)u^*w^*=\Theta_t^{(2)}(x)\qquad\qquad\text{ for all } x\in\phi(N)
$$
as required.
\end{proof}

Recall that we have already fixed a *isomorphism $\Phi:R\bar\otimes X\to X$ and we have denoted by $\Phi_\omega:(R\bar\otimes X)^\omega\to X^\omega$ the induced component-wise *isomorphism.

\begin{definition}\label{def:expansion}
Let $\phi:N\to(R\bar\otimes X)^\omega$. For each $x\in N$, let $(X_i^\phi)\in\ell^\infty(R\bar\otimes X)$ be a lift of $\phi(x)$. Define $1\otimes\phi$ through the following diagram
$$
\xymatrix{(1\otimes X_n^\phi)_n\in\ell^\infty(R\bar\otimes R\bar\otimes X)\ar[r]\ar[d]^{\oplus_{\mathbb N}(1\otimes\Phi)} & (R\bar\otimes R\bar\otimes X)^\omega\ar[d]^{(1\otimes\Phi)_\omega}\\
\ell^\infty(R\bar\otimes X)\ar[r] & (R\bar\otimes X)^\omega}
$$
i.e. $(1\otimes\phi)(x)$ is the image of the element $(1\otimes X_n^\phi)_n\in\ell^\infty(R\bar\otimes R\bar\otimes X)$ down in $(R\bar\otimes X)^\omega$.
\end{definition}

Exactly as in Lemma 3.2.3 in \cite{Br}, we get the following 

\begin{lemma}\label{lem:expansion}
For all $\phi:N\to(R\bar\otimes X)^\omega$, one has $[1\otimes\phi]=[\phi]$.
\end{lemma}

\begin{lemma}\label{lem:compositionisomorphisms}
Let $\theta_s,\theta_t$ be two standard isomorphisms. Then
$$
\theta_s\circ\theta_t:X^\omega\to\theta_s(p_t)X^\omega\theta_s(p_t)
$$
is still a standard isomorphism.
\end{lemma}

\begin{proposition}\label{prop:distributive}
For all $s,t>0$ and $[\phi],[\psi]\in\mathbb Hom(N, (X^\omega)^\infty)$, the following properties are satisfied:
\begin{enumerate}
\item $0[\phi]=0$,
\item $1[\phi]=[\phi]$,
\item $s(t[\phi])=(st)[\phi]$,
\item $s([\phi]+[\psi])=s[\phi]+s[\psi]$,
\item if $s+t\leq1$, then $(s+t)[\phi]=s[\phi]+t[\phi]$.
\end{enumerate}
\end{proposition}

\begin{proof}
The first two properties are trivial. The third property follows by Lemma \ref{lem:compositionisomorphisms} and Proposition \ref{prop:independent}. The fourth property can be easily proved by direct computation. Let us prove the fifth property. Fix $n>(s+t)\tau_\infty(\phi(1))$ and twist $\phi$ by a unitary in such a way that $\phi(N)\subseteq M_n(\mathbb C)\otimes X^\omega=(M_n(\mathbb C)\otimes X)^\omega$, since $M_n(\mathbb C)$ is finite dimensional. Now, $M_n(\mathbb C)$ has a unique unital embedding into $R$ up to unitary equivalence and therefore we may suppose that $\phi(N)\subseteq(R\bar\otimes X)^\omega$ and we may apply the construction in Definition \ref{def:expansion} and Lemma \ref{lem:expansion} to replace $[\phi]$ with $[1\otimes\phi]$. Now we have the freedom to choose orthogonal projections of the form
$$
p_s\otimes1\otimes1, \quad p_t\otimes1\otimes1,\quad (p_s+p_t)\otimes1\otimes1\in(R\bar\otimes R\bar\otimes X)^\omega
$$
and use these projections to define standard isomorphisms. It is then clear that
$$
\Theta_s\circ(1\otimes\phi)+\Theta_t\circ(1\otimes\phi)=\Theta_{t+s}\circ(1\otimes\phi)
$$
which implies that $[\Theta_s\circ\phi]+[\Theta_t\circ\phi]=[\Theta_{s+t}\circ\phi]$; i.e. $s[\phi]+t[\phi]=(s+t)[\phi]$.
\end{proof}

We show in an appendix that the five algebraic conditions above imply $\mathcal{G}(N, X^\omega)$ inherits a natural vector space structure.  Furthermore, the metric on $\Hom(N, (X^\omega)^\infty)$ extends to a norm on $\mathcal{G}(N, X^\omega)$ and even makes it a Banach space (see \cite{Ca-Fr} for details).  In summary: 

\begin{theorem} If $N$ is separable and $X$ is McDuff, then $\mathcal{G}(N, X^\omega)$ has a Banach space structure with canonical actions of $Out(N)$ and $Out(X^\omega \bar{\otimes} B(H))$. 
\end{theorem} 

Since the embedding $X^{\omega} \hookrightarrow (X^\omega)^\infty, x \mapsto (1 \otimes e_{11}) (x \otimes 1)  (1 \otimes e_{11})$ gives rise to an embedding $\Hom (N, X^\omega) \hookrightarrow \Hom(N, (X^\omega)^\infty)$ which is evidently compatible with the ``convex-like" structure introduced in \cite{Br}, we have a new and more concrete proof of the vector-space embedding that motivated \cite{Ca-Fr}. 

\begin{corollary} If $N \subset \R$ is a nonamenable separable subfactor, then the non-second-countable, complete metric space $\Hom(N,\R)$ is affinely and isometrically isomorphic to a closed convex subset of a Banach space. 
\end{corollary}

\section*{Appendix} 

Here we establish a purely algebraic result which is surely known to algebraists, though we're unaware of a reference.  Namely, we consider conditions that imply the Grothendieck group of an abelian monoid is a vector space.

Assume we have a commutative and cancelative monoid $G_+$ equipped with an action $[0,1]\curvearrowright G_+$ satisfying the following properties. For all $g,g_1,g_2\in G_+$ and for all $s,t\in[0,1]$,

\begin{enumerate}
\item $0.g=0$,
\item $1.g=g$,
\item if $s+t\leq1$, then $(s+t).g=s.g+t.g$,
\item $s.(g_1+g_2)=s.g_1+s.g_2$,
\item $(st).g=s.(t.g)$.
\end{enumerate}

Let $t>0$, denote by $f_t$ the floor of $t$, that is the largest integer smaller than or equal to $t$, and denote by $d_t=t-f_t$ the decimal part of $t$. Having an action $[0,1]\curvearrowright G_+$, we can easily define an action $\mathbb R_+\curvearrowright G_+$ by setting
$$
t.g=f_t.g+d_t.g
$$
where $f_t.g$ is just the $f_t$-fold sum $g+\ldots+g$.

\begin{proposition}\label{prop:positivescalarmultiplication}
The action $\mathbb R_+\curvearrowright G_+$ satisfies the same five properties as above (minus the restriction in (3) that $s+t\leq1$, of course).
\end{proposition}

\begin{proof}
The first two properties are trivial as well as the fourth one. Let us prove the third property. We have to prove that
\begin{align}\label{eq:firstversion}
f_{s+t}.g+d_{s+t}.g=f_s.g+d_s.g+f_t.g+d_t.g
\end{align}
expanding the terms of the form $n.g$, with $n\in\mathbb N$, the previous equality can be rewritten as follows
\begin{align}\label{eq:secondversion}
g+\ldots+g+d_{s+t}.g=g+\ldots+g+d_s.g+g+\ldots+g+d_t.g
\end{align}
Observe that the sum is commutative and therefore the $g$'s with coefficient 1 can be put wherever we want. This will be important to apply the following argument. Suppose that $d_{s+t}>d_t$. Take the $g$ closest to $d_t$ and rewrite it as
$$
g=(1-(d_{s+t}-d_t)+d_{s+t}-d_t).g
$$
Using the third property above we can rewrite Equation (\ref{eq:secondversion}) as follows
\begin{align}\label{eq:thirdversion}
g+\ldots+g+d_{s+t}.g=g+\ldots+g+d_s.g+g+\ldots+g+(1-d_{s+t}+d_t).g+d_{s+t}.g
\end{align}
Since the monoid is cancelative, the last terms cancel out. It is clear that we can iterate this procedure and, since the sum of the coefficients of the $g$'s on the left-hand side is equal to the sum of the coefficients of the $g$'s on the right-hand side, we end in an identity $0=0$. This means that the starting equality in Equation (\ref{eq:firstversion}) holds, as desired.\\
Now we prove the fifth property. Observe that it is true if one between $s$ and $t$ belongs to $\mathbb N$, by definition. Using this observation and using the third and the fourth property, we have
\begin{align*}
s.(t.g)=\\
&=s.(f_t.g+d_t.g)=\\
&=f_s.(f_t.g+d_t.g)+d_s.(f_t.+d_t.g)=\\
&=(f_sf_t).g+(f_sd_t).g+(d_s.f_t).g+(d_sd_t).g=\\
&=(f_sf_t+f_sd_t+d_s.f_t+d_sd_t).g=\\
&=(f_sf_t+f_st-f_sf_t+sf_t-f_sf_t+st-f_st-sf_t+f_sf_t).g=\\
&=(st).g
\end{align*}
\end{proof}

Let us recall the definition of the Grothendieck group of an abelian monoid. Given an abelian monoid $G_+$, its Grothendieck group is the abelian group constructed as follows:
\begin{itemize}
\item Consider in $G_+\times G_+$ the equivalence relation
$$
(g_1,g_2)\sim(h_1,h_2)\,\,\,\,\,\,\,\,iff\,\,\,\,\,\,\,\,g_1+h_2=h_1+g_2
$$
\item Let $G=(G_+\times G_+)/\sim$ equipped with the component-wise
operation, that is well-defined on the equivalent classes.
\end{itemize}
$G$ is an abelian group and, in general, $G_+$ does not embed into $G$. If $G_+$ is a cancelative monoid, then $G_+$ embeds into $G$.

Notice that by definition, the class
$[(g_1,g_2)]$ represents the element $g_1-g_2$ and the inverse of $[(g_1,g_2)]$ is $[(g_2,g_1)]$.

\begin{proposition}\label{lem:vectorspace}
Let $G_+$ be an abelian, cancelative monoid equipped with an action
$\mathbb R_+\curvearrowright G_+$ such that for all
$s,t\in\mathbb R_+$ and $g,g_1,g_2\in G_+$
\begin{enumerate}
\item \begin{color}{red}$0g=0$\end{color}
\item $1g=g$
\item $s(tg)=(st)(g)$
\item $t(g_1+g_2)=tg_1+tg_2$
\item $(s+t)g=sg+tg$
\end{enumerate}
Then the Grothendieck group of $G_+$ is a vector space with
scalar multiplication $s[(g_1,g_2)]=[(sg_1,sg_2)]$, when $s\geq0$
and $s[(g_1,g_2)]=[((-s)g_2,(-s)g_1)]$, when $s<0$.
\end{proposition}

\begin{proof}
We have to prove the following properties
\begin{enumerate}
\item $0[(g_1,g_2)]=[(0,0)]$
\item $1[(g_1,g_2)]=[(g_1,g_2)]$
\item $(s+t)[(g_1,g_2)]=s[(g_1,g_2)]+t[(g_1,g_2)]$
\item $s(t[(g_1,g_2)])=(st)[(g_1,g_2)]$
\item $t([(g_1,g_2)]+[(h_1,h_2)])=t[(g_1,g_2)]+t[(h_1,h_2)]$
\end{enumerate}
The first two properties are trivial, as is the third one when $s,t\geq0$. Let us consider the other cases.
\begin{itemize}
\item If $s,t\leq0$, one has
\begin{align*}
(s+t)[(g_1,g_2)]\\
&=-[(-(s+t))g_1,(-(s+t))g_2)]\\
&=-[((-s-t)g_1,(-s-t)g_2)]\\
&=-[(-sg_1-tg_1,-sg_2-tg_2)]\\
&=-[(-sg_1,-sg_2)]+[(-tg_1,-tg_2)]\\
&=s[(g_1,g_2)]+t[(g_1,g_2)]
\end{align*}
\item If $s\geq 0, t\leq0$ and $s+t\geq0$, one has
$$
(s+t)[(g_1,g_2)]=[((s+t)g_1,(s+t)g_2)]
$$
and
$$
s[(g_1,g_2)]+t[(g_1,g_2)]=[(sg_1+(-t)g_2,sg_2+(-t)g_1)]
$$
and these two classes are indeed equal:
\begin{align*}
(s+t)g_1+sg_2+(-t)g_1\\
&=(s+t)g_1+(s+t)g_2+(-t)g_2+(-t)g_1\\
&=(s+t)g_2+(s+t)g_1+(-t)g_1+(-t)g_2\\
&=(s+t)g_2+sg_1+(-t)g_2
\end{align*}
\item The case $s\geq0, t\leq0, s+t\leq0$ is similar.
\item The remaining cases follow by symmetry.
\end{itemize}
The fourth property is also  trivial when $s,t\geq0$. Let us consider the other
cases
\begin{itemize}
\item If $s\geq0$ and $t<0$, then
\begin{align*}
s(t[(g_1,g_2)])\\
&=s[((-t)g_2,(-t)g_1)]\\
&=[((-st)g_2,(-st)g_1)]\\
&=(-(st))[(g_2,g_1)]\\
&=(st)[(g_1,g_2)]
\end{align*}
\item The case $s<0$ and $t\geq0$ is the same.
\item If $s,t\leq0$, one has
\begin{align*}
s(t[(g_1,g_2)])\\
&=s[(-t)g_2,(-t)g_1]\\
&=[((-s)(-t)g_1,(-s)(-t)g_2)]\\
&=(st)[(g_1,g_2)]
\end{align*}
\end{itemize}
The fifth property is trivial when $t\geq0$, so let us suppose $t<0$. One has
\begin{align*}
t([(g_1,g_2)]+[(h_1,h_2)])\\
&=t[(g_1+h_1,g_2+h_2)]\\
&=[((-t)(g_2+h_2),(-t)(g_1+h_1))]\\
&=[((-t)g_2+(-t)h_2,(-t)g_1+(-t)h_1)]\\
&=[(-t)g_2,(-t)g_1]+[(-t)h_2,(-t)h_1]\\
&=t[(g_1,g_2)]+t[(h_1,h_2)]
\end{align*}
\end{proof}


\begin{thebibliography}{9}
\bibitem[Br]{Br} N.P. Brown, \emph{Topological Dynamical Systems Associated to
$II_1$-factors}, Advances in Mathematics, vol. 227
(4) (2011), 1665-1699.
\bibitem[Br2]{Br2} N.P. Brown, \emph{Connes' embedding problem and Lance's
WEP}, Int. Math. Rev. Notices (2004) 10, 501-510.
\bibitem[Ca-Fr]{Ca-Fr} V.Capraro and T.Fritz, \emph{On the axiomatization of convex subsets of a Banach space}, To appear in Proc. Amer. Math. Soc.
\bibitem[Co]{Co} A. Connes, \emph{Classification of injective factors}, Ann. of Math. 104
(1976), 73-115.
\bibitem[CJ]{CJ} A. Connes and V. Jones, \emph{Property T for von Neumann algebras}, Bull. London Math. Soc. 17 (1985), 57-62. 

\bibitem[Ka-Ri2]{Ka-Ri2} R.V.Kadison and J.R.Ringrose, \emph{Fundamentals of the theory of operator algebras, vol.
II}
\bibitem[PV]{PV} S. Popa and S. Vaes, \emph{On the fundamental group of II$_1$-factors and equivalence relations arising from group actions}, Quanta of Maths, Clay Mathematics Institute Proceedings 11 (2011), pp. 519-541.
\bibitem[Sh]{Sh} D. Sherman, \emph{Notes on automorphisms of ultrapowers of
$II_1$-factors}, Studia Mathematica 195 (2009) 201-217. W.
\end{thebibliography}
\end{document}